\documentclass[11pt]{amsproc}
\usepackage{amstext,amsmath,amssymb,amsthm}
\usepackage{latexsym}
\usepackage{exscale}
\usepackage{color}

\RequirePackage{srcltx}

\numberwithin{equation}{section}

\newcommand\h[1]{\mkern2mu\widehat{\mkern-2mu#1}}
\newcommand\on[1]{|\mkern-1mu|\mkern-1mu|#1|\mkern-1mu|\mkern-1mu|}

\renewcommand\Re{\operatorname{Re}}

\newcommand\R{\mathbb{R}}
\newcommand\N{\mathbb{N}}

\newcommand\Sp{\mathbb{S}}
\newcommand\cL{\mathcal{L}}

\newtheorem{Thm}{Theorem}[section]
\newtheorem{Lemma}[Thm]{Lemma}
\newtheorem{Cor}[Thm]{Corollary}
\newtheorem{Prop}[Thm]{Proposition}
\theoremstyle{remark}
\newtheorem{Rem}{Remark}[section]

\begin{document}

\title[Pitt inequalities and restriction theorems]{Pitt inequalities and
restriction theorems for~the~Fourier transform}
\author{Laura De Carli}
\address{L.~De Carli, Florida International University,
Department of Mathematics,
Miami, FL 33199, USA}
\email{decarlil@fiu.edu}
\author{Dmitriy~Gorbachev}
\address{D.~Gorbachev, Tula State University,
Department of Applied Mathematics and Computer Science,
300012 Tula, Russia}
\email{dvgmail@mail.ru}
\author{Sergey~Tikhonov}
\address{S. Tikhonov, ICREA, Centre de Recerca Matem\`{a}tica, and UAB\\
Campus de Bellaterra, Edifici~C
08193 Bellaterra (Barcelona), Spain.}
\email{stikhonov@crm.cat}
\thanks{D.~G. was supported by the RFBR (no.~13-01-00045), the Ministry of
Education and Science of the Russian Federation (no.~5414GZ), and by D.~Zimin's
Dynasty Foundation. S.~T. was partially supported by grant 2014-SGR-289 from
AGAUR (Generalitat de Catalunya) and RFBR 13-01-00043.}
\begin{abstract}
We prove  new  
Pitt inequalities 
for the Fourier transforms
with radial and non-radial weights
using weighted restriction inequalities  for the Fourier transform on the sphere.  
We also prove new Riemann--Lebesgue estimates and versions of the uncertainty
principle for the Fourier transform.
\end{abstract}
\maketitle

\section{Introduction}

Weighted inequalities for the Fourier transform provide a natural balance
between functional growth and smoothness. On $\R^n$ it is important to
determine quantitative comparisons between the relative size of a function and
its Fourier transform at infinity. We will let $\h{f}(\xi) =\int_{\R^n}
e^{ix\xi}f(x)\,dx$, $\xi\in \R^{n}$, be the Fourier transform in $L^1(\R^n)$,
and $\|{\,\cdot\,}\|_p$ be the standard norm in $L^p(\R^n)$. We consider Pitt
type inequalities
\begin{equation}\label{1e-Genpitt}
\|u^{\frac{1}{q}} \h{f}\|_q\le C\|v^{\frac{1}{p}} f\|_p, \quad f\in
C^\infty_0(\R^n).
\end{equation}

Here and throughout the paper, $u$ and $v$ are non-negative measurable
functions on $\R^n$, and $1\le p, q\le \infty$ unless otherwise specified. We
will use~$C$ to denote numeric constants that may change from line to line. We
will let $p'=\frac{p}{p-1}$ be the conjugate exponent of $1\le p\le \infty$,
and we will often let $x= \rho\omega$, with $\omega\in \Sp^{n-1}$ and
$\rho=|x|$. We denote by $|E|$ the Lebesgue measure of a set $E$ and by
$\chi_E(x)$ be the characteristic function of $E$.

In 1983, Heinig \cite{H}, Jurkat--Sampson \cite{JS} and Muckenhoupt \cite{M},
\cite{M1} proved
\begin{Thm}\label{heinig} { Let $n\ge 1$.}
If the weights $u$ and $v$ satisfy
\begin{equation}\label{1e-weight-cond}
\sup_{s>0}\left(\int_{0}^s u^*(t )\,dt\right)^{\frac{1}{q}}
\left(\int_{0}^{\frac{1}{s}} [(1/v)^*(t)]^{\frac{1}{p-1}}\,dt\right)^{\frac{1}{p'}}=C<\infty,
\end{equation}
for $1<p\le q<\infty$, where $g^*$ is the non-increasing rearrangement of $g$,
then \eqref{1e-Genpitt} holds.
\end{Thm}

To formulate necessary conditions for inequality \eqref{1e-Genpitt} to hold we
recall the definition of polar set. If $A\subset \R^n$,
\[
A^{*}=\big\{\xi\in \R^{n}\colon |x\xi|\le 1,\ x\in A\big\}
\]
is the polar set of~$A$ (see {\cite[\S\:4]{St}}). We prove the following

\begin{Thm}\label{nec} Let $n\ge 1$.
Suppose that the Pitt inequality \eqref{1e-Genpitt} holds for any $f\in
C^\infty_0(\R^n)$ and for $1<p,q < \infty$.

\smallbreak
\textup{(1)} \ Let a convex body $A\subset \R^{n}$ be centrally symmetric with
respect to the origin. Then
\begin{equation}\label{nec-cond}
\sup_{A}\left(
\int_{c A^{*}}u(\xi)\,d\xi\right)^{\frac 1 q}\left(\int_{A}v^{1-p'}(x)\,dx\right)^{\frac 1{p'}}=\
C<\infty,
\end{equation}
where $c <\pi/2$ and $A^{*}$ is a polar set of the set~$A$.

\smallbreak
\textup{(2)} \ Let the weights $u(x)=u_0(|x|)$ and $v(x)=v_0(|x|)$ be radial,
then
\begin{equation}\label{1e-weight-cond-2}
\sup_{s>0}\left(\int_{|x|<s} u(x)\,dx\right)^{\frac 1 q}
\left(
\int_{|x|<\frac{c_n}{s}}v^{1-p'}(x)\,dx\right)^{\frac 1{p'}}=C<\infty,
\end{equation}
where $c_n$ is any positive number less than $q_{n/2-1}$, the first zero of
the Bessel function $J_{n/2-1}(t)$. In particular, $q_{n/2-1}\ge \pi/2$.

\smallbreak
\textup{(3)} Results of the part \textup{(1)} also hold if one replaces the
sets $A$ and $c A^{*}$ by a union of their disjoint translations, that is, by
the sets $A_1=\bigcup_{j=1}^{N_1} (A+x_j)$ and
$A_2=\bigcup_{j=1}^{N_2} (cA^*+\xi_j)$ for any $x_j$ and $\xi_j$.
\end{Thm}

Note that in this theorem we do not assume $q \ge p$.

Part (2) of the theorem is known with a smaller constant $c$; see the
proof of Theorem 3.1 in \cite{H}. Moreover, part (3) generalizes the following
necessary condition (see \cite[Th. 3]{berndt}):
\begin{equation}\label{1e-Necessarycondition}
\left(\int_{Q_1}u(\xi)\,d\xi\right)^{\frac 1 q}\left(\int_{Q_2}v^{1-p'}(x)\,dx\right)^{\frac{1}{p'}}=
C<\infty,
\end{equation}
for all cubes $Q_1$ and $Q_2$ such that $|Q_1|\,|Q_2|=1$.

We should also mention \cite[Theorem 2.1]{L} where a necessary condition similar to
\eqref{1e-Necessarycondition}, with $u$ replaced by a measure $d\mu$, was proved.

\medskip When $u(x)=u_0(|x|)$ and $v(x)=v_0(|x|)$ are radial, with $u_0(\cdot)$
non-increasing and $v_0(\cdot)$ non-decreasing, then \eqref{1e-weight-cond-2}
is necessary and sufficient for the validity of \eqref{1e-Genpitt} (see
\cite{H}). In particular, if $u(x)$ and $v(x)$ are locally integrable power
weights, i.e., in the form of $u= |x|^b$ and $v=|x|^a$, with $a,b>-n$, we get
that the classical Pitt inequality
\begin{equation}\label{1e-pitt}
\left(\int_{\mathbb{R}^n}
|\h{f}(\xi)|^q |\xi|^b\,d\xi\right)^{\frac 1 q}\le C
\left(\int_{\mathbb{R}^n}
|f(x)|^p|x|^a\,dx\right)^{\frac 1 p},
\end{equation}
where $f\in C^\infty_0(\R^n)$,
holds if and only if
\begin{equation}\label{1e-relation-ab}
  \frac{a}{p}+\frac{b}{q} = n\left(1-\frac{1}{p}-\frac{1}{q}\right),
\end{equation}
\begin{equation}\label{1e-relation-b}
-n<b\le 0, \quad\text{and}\quad 0\le a<n(p-1);
\end{equation}
see \cite{Pitt,Stein,BH}.

\medbreak
Pitt type inequalities with power weights that satisfy less restrictive
conditions than those in \eqref{1e-relation-b} are only valid on special
subspaces of $L^p(\R^n)$. We have proved in \cite{DGT} that if $f$ is a product
of a radial function and a spherical harmonics of degree $k\ge 0$, then
\eqref{1e-Genpitt} is satisfied with $u=|x|^a$ and $v=|x|^b$ if and only if $a$
and~$b$ satisfy \eqref{1e-relation-ab} and
\begin{equation}\label{1e-relation-bb}
(n-1)\left(\frac{1}{2}-\frac{1}{p}\right) +\max\left\{ \frac{1}{p'}- \frac{1}{q}, \,
0\right\}\le \frac{b}{p} < \frac{n}{p'}+k,
\tag{\theequation$'$}
\end{equation}
{
which is less restrictive than the conditions in \eqref{1e-relation-b} even for $k= 0$ }

\medbreak
In this paper we prove $L^p$--$L^q$ Pitt inequalities for radial and non-radial
weights $u$ and $v$. Our main tools are weighted restriction inequalities for
the Fourier transform in $\R^n$, $n\ge 2$. That is,
\begin{equation}\label{1e-wRestrIneq}
\left(\int_{\Sp^{n-1}} |\h{f}(\omega)|^q
U(\omega)\,d\sigma(\omega)\right)^{\frac{1}{q}}\le C\left(\int_{\R^n}| f(x)|^p
v(x)\,dx\right)^{\frac{1}{p}},
\end{equation}
where $U$ and $v$ are non-negative and measurable on $\Sp^{n-1}$ and
$\R^n$, respectively, and $f\in C^\infty_0(\R^n)$.

We recall several known
restriction theorems in Section~2. In Section~3 we present new Pitt inequalities using
restriction inequalities. In particular, we prove the following
\begin{Thm}\label{new}
Let $ 1\le p < \frac{2(n+2)}{n+4}$ and $1\le q\le \frac{n-1}{n+1}\,p'$, with $n\ge 2$.
Suppose that $u(x)=u_0(|x|)$ satisfies
\begin{equation}\label{1e-cond-nu}
\int_{0}^\infty \rho^{n-1-\frac{qn}{p'}} u_0(\rho)\,d\rho <\infty.
\end{equation}
Then for every $f\in C^\infty_0(\R^n)$
\begin{equation}\label{1e-NewPitt}
\|u^{\frac{1}{q}} \h{f}\|_q\le
C\|f\|_p.
\end{equation}
\end{Thm}

\begin{Rem}
(\textit{i}) \ The proof of Theorem \ref{1T-Restr-Pitt}, of which Theorem
\ref{new} is a special case, shows that the constant~$C$ in \eqref{1e-NewPitt}
equals $C'\left(\int_{0}^\infty \rho^{n-1-\frac{qn}{p'}}
u_0(\rho)\,d\rho\right)^{{1}/{q}}$, where $C'$ depends on~$n,\ p,\ q$.
\\
(\textit{ii}) \ When $u\in L^p(\R^n)$ with $1\leq p\leq 2$ and $ q=1$, \eqref{1e-NewPitt} is
valid also when $u$ is not radial; indeed, by Hausforf--Young inequality,
\[
\|u \h{f}\|_ 1\le \|u \|_ p \|\h{f}\|_{p'}\leq
C\|f\|_p.
\]
(\textit{iii}) \ Theorem \ref{heinig} and most of the Pitt inequalities in the
literature are proved for $1<p\le q<\infty$. Theorem \ref{new} provides a
rather simple sufficient condition for (\ref{1e-NewPitt}) that applies
either when $p\le q$ or $p>q$.
Note that the known sufficient conditions 
for (\ref{1e-NewPitt}) 
are usually quite difficult to verify especially in the case $p>q$ (see for
example \cite{BH}).

\end{Rem}

Theorem \ref{new} applies in cases where Theorem \ref{heinig} does not: In
Section~4 we construct a radial weight $u$ for which the inequality
\eqref{1e-weight-cond} does not hold, but \eqref{1e-cond-nu} holds for $u_0$
and therefore \eqref{1e-NewPitt} is valid.

\medskip
The rest of the paper is organized as follows. In Section 5 we prove
necessary conditions for the Pitt inequality \eqref{1e-Genpitt} to hold
(Theorem\ref{heinig}), necessary conditions for the weighted restriction
inequality \eqref{1e-wRestrIneq} to hold (Proposition~\ref{P-nec1}), and
sufficient conditions from Section~3. These are the main results of the paper.

In Section~6 we prove new versions of the uncertainty principle for the Fourier
transform.

In Section~7 we apply our new Pitt's inequality to get a quantitative version
of the Riemann--Lebesgue lemma, which provides an interrelation between the
smoothness of a function and the growth properties of the Fourier transforms.

\medskip
Finally, we would like to mention make the following interesting observation which perhaps is
not new: the Pitt inequality \eqref{1e-Genpitt} holds if and only if, for some
$s\ge p$, we have $\|u^{\frac{1}{q}} \h{f}\|_q\le
C\|w^{-1}\|_{\frac{p}{s-p}}^{\frac{1}{s}}\|w^{\frac{1}{s}}v^{\frac{1}{p}} f\|_s $ whenever
$w^{-1}\in L^{\frac{p}{s-p}}(\R^n)$. In particular, the inequality $\|\h{f}\|_{p'}
\le C \|w^{\frac{1}{s}} f\|_s $ holds for every
$1\le p\le 2$ whenever $w^{-1}\in L^{\frac{p}{s-p}}(\R^n)$, $s\ge p$.
We will prove this fact
in Section 5.

\section{Restriction theorems for the Fourier transform}

The Tomas--Stein restriction inequality for the Fourier transform on the unit
sphere states that, for every $f\in C^\infty_0(\R^n)$, $n\ge 2$,
\begin{equation}\label{2eRestrIneq}
\left(\int_{\Sp^{n-1}} |\h{f}(\omega)|^q\,d\sigma(\omega)\right)^{\frac{1}{q}}
\le C\left(\int_{\R^n} |f(x)|^p\,dx\right)^{\frac{1}{p}},
\end{equation}
where $d\sigma(\omega)$ is the induced Lebesgue measure on $\Sp^{n-1}$,
$1\le q\le \frac{n-1}{n+1}\,p'$, and $1\le p\le \frac{2(n+1)}{n+3}$
\cite{T,St}.

The same inequality holds also if $d\sigma(\omega)$ is replaced by
$\chi(\omega)\,d\sigma(\omega)$, with $\chi\in C_0^\infty(\Sp^{n-1})$
\cite{St}. So, if $T(f)=\h{f}|_{\Sp^{n-1}}$ is the restriction operator, $T$
maps $L^p(dx)$ into $ L^q(d\sigma)$ boundedly when $p$, $q$ are as in the
Tomas--Stein theorem.

Note that \eqref{2eRestrIneq} is trivial when $p=1$ because
\[
\left(\int_{\Sp^{n-1}} |\h{f}(\omega)|^q\,d\sigma(\omega)\right)^{\frac{1}{q}}\le
\omega_{n-1}^{\frac{1}{q}}\|\h{f}\|_\infty\le \omega_{n-1}^{\frac{1}{q}}\|f\|_1,
\]
where $\omega_{n-1}=|\Sp^{n-1}|$.

The restriction conjecture states that
inequality \eqref{2eRestrIneq} is valid for all $1\le q\le
\frac{n-1}{n+1}\,p'$ and $1\le p < \frac{2n}{n+1}$. When $n=2$ the restriction
conjecture has been proved by C.~Fefferman \cite{F}. When $n\ge 3$, T.~Tao
\cite{Ta} has proved that~\eqref{2eRestrIneq} is valid for $1\le
p<\frac{2(n+2)}{n+4}$. Note that
$\frac{2(n+2)}{n+4}=\frac{2n}{n+1}$ when~$n=2$.

\medbreak
Weighted versions of the restriction inequality \eqref{2eRestrIneq} in the form
of
\begin{equation}\label{2eRestrIneq-w}
\left(\int_{\Sp^{n-1}}|\h{f}(\omega)|^q\,U(\omega)\,d\sigma(\omega)\right)^{\frac{1}{q}}\le
C\left(\int_{\R^n} |f(x)|^p v(x)\,dx\right)^{\frac{1}{p}}
\end{equation}
have been proved by several authors. In most of the theorems in the literature,
$1\le p\le q\le\infty$ and $U(\omega)$ is the restriction of a function
$\widetilde U(x)\in C^\infty (\R^{n})$, often with compact support.

\medskip
The following duality argument will be used in the proof of the theorems in the
next section. The technique is well known, but we state and prove Lemma
\ref{L-duality} in this paper for the sake of completeness.

\begin{Lemma}\label{L-duality}
Assume $U(x/|x|)=U(\omega)\in L^{1}(\Sp^{n-1})$. Inequality
\eqref{2eRestrIneq-w} is equivalent to
\begin{equation}\label{2eduality}
\left\|\int_{\Sp^{n-1}}g(\omega) e^{i\omega y}
U^{\frac{1}{q}}(\omega)\,d\sigma(\omega)\right\|_{L^{p'}(v^{1-p'}dy)}
\le
C \|g\|_{L^{q'}(\Sp^{n-1})}.
\end{equation}
\end{Lemma}

\medskip

In Section 5 we prove necessary conditions for the weighted restriction
inequality \eqref{2eRestrIneq-w} to hold. To the best of our knowledge these
results are new.

\begin{Prop}\label{P-nec1}
Assume that the inequality \eqref{2eRestrIneq-w} holds with $U^{1-q'}
(\omega)\in L^{1 }(\Sp^{n-1})$. Then
\begin{equation}\label{v-j-nec-cond-1}
\int_{\R^n}v^{1-p'}(x)|j_{n/2-1}(|x|)|^{p'}\,dx<C,
\end{equation}
where $j_{\alpha}(t)=\Gamma(\alpha+1)(t/2)^{-\alpha}J_{\alpha}(t)$ is the
normalized Bessel function.
\end{Prop}

A~special case of \eqref{v-j-nec-cond-1} is in \cite[(3.1)]{BS}.
In particular, we obtain the following result.
\begin{Cor}\label{C-nec2}
Assume that the inequality \eqref{2eRestrIneq-w} holds with $U^{1-q'}
(\omega)\in L^{1 }(\Sp^{n-1})$; assume $v$ radial and non-negative, and that
$v(x)=v_0(|x|)$ satisfies either
\begin{equation}\label{vs1}
\int_{ A } v_{0}^{1-p'}(t-|A|)\,dt\le C\int_{A} v_{0}^{1-p'}(t)\,dt,
\end{equation}
or
\begin{equation}\label{vs2}
\int_{ A } v_{0}^{1-p'}(t+|A|)\,dt\le C\int_{A} v_{0}^{1-p'}(t)\,dt,
\end{equation}
 for all finite intervals $A$, with a constant $C$ independent of $A$.
Then
\[
\int_{\R^{n}}v^{1-p'}(x)(1+|x|)^{-\frac{p'(n-1)}2}\,dx<C.
\]
\end{Cor}

\begin{Rem}
If $v_0^{1-p'}$ satisfies a doubling condition, that is,
\[
\int_{2A}v_{0}^{1-p'}(t)\,dt\le C\int_{A} v_{0}^{1-p'}(t)\,dt,
\]
for all intervals $A$, where
$2A$ is the
interval twice the length of $A$ and with the midpoint coinciding with that of
$A$, then both \eqref{vs1} and \eqref{vs2} hold. If $v_0$ is monotonic, then
at least one of the conditions \eqref{vs1} and \eqref{vs2} hold.
\end{Rem}

\medskip
Weighted restriction theorems were
intensively studied
for piecewise power weights, i.e. in the form of
\begin{equation}\label{p-w1}
v(x)=\begin{cases} |x|^\alpha, & \textup{if} \ |x|\le 1, \\ |x|^\beta, &
\textup{if} \ |x|>1, \end{cases}
\end{equation}
see e.g. \cite{BS}.
The method of the proof of \cite[Cor.~2.8]{CS} can be used to prove the
following

\begin{Lemma}\label{2L-restr-gen-CS}
Let $d\mu$ and $d\nu$ be measures on $\R^n$, $n\ge 1$, and let $1\le p\le q$ and $s\ge
p$. An operator $T$ maps $L^p(d\mu)\to L^q(d\nu)$ boundedly if and only $T$
maps $L^s(w\,d\mu)\to L^q(d\nu)$ boundedly whenever $w^{-1}\in
L^{\frac{p}{s-p}}(d\mu)$ and
\[
\on{T}_{L^s(w\,d\mu)\to L^q(d\nu)}\le
C\|w^{-1}\|_{L^{\frac{p}{s-p}}(d\mu)}^{\frac{1}{s}}.
\]
\end{Lemma}

The proof is in Section~5. If we apply Lemma \ref{2L-restr-gen-CS} to the
restriction operator, with the the Tomas--Stein exponents $s=q=2$ and $p
=\frac{2(n+1)}{n+3}$ we require $w^{-1}\in L^{\frac{n+1}{2}}(\R^n)$. This
condition applied to piecewise power weight,
allows $\alpha <\frac{2n}{n+1} $ and $\beta
> \frac{2n}{n+1} $.

\medbreak
These exponents are not sharp: S.~Bloom and G.~Sampson have proved in \cite{BS}
a number of restriction theorems with piecewise power weights, and have
obtained, in most cases, sharp conditions on $\alpha$ and $\beta$. One of the
results in \cite[Thm. 5.6]{BS} is the following

\begin{Thm}\label{2T-Bloom}
Let $1 < p\le 2 $, $n\ge 2$, $2\le q\le \frac{n-1}{n+1}\,p'$. Let $v(x)$ given by
\eqref{p-w1}. Then \eqref{2eRestrIneq-w} with $U =1$ holds if and only if
$\alpha <n(p-1)$ and $\beta \ge 0$. Moreover, \eqref{2eRestrIneq-w} holds with $p=q=2$ also when $U= 1$ and
$v(x)$ is as in \eqref{p-w1} with $\alpha<n$ and $\beta>1$.
\end{Thm}
\medskip

We also notice that weighted restriction theorems have been proved for weights
in the Campanato--Morrey spaces: for $0\le \alpha\le \frac{n}{r}$ and $r\ge 1$,
the Campanato--Morrey space $\cL^{\alpha,r}$ is defined as
\[
\cL^{\alpha,r}=\Biggl\{ f\in L^r_\text{loc}(\R^n) \colon\|f\|_{r,
\alpha}=\sup_{\substack{ x\in\R^n \\ \rho >0 } }
\rho^\alpha\biggl(\rho^{-n}\int_{|y-x|<\rho}
|f(y)|^r\,dy\biggr)^{\frac{1}{r}}<\infty\Biggr\}.
\]
Note that $\cL^{\alpha,\frac{n}{\alpha}}= L^{\frac{n}{\alpha}}(\R^n)$ and
$\cL^{0,r}(\R^n)=L^{\infty}(\R^n)$.

A.~Ruiz and L.~Vega have proved in \cite{RV} the following

\begin{Thm}\label{T-RuizV}
Suppose that $V\in \cL^{\alpha,r}$, with $\frac{\alpha}{n}\le \frac{1}{r}
<\frac{2(\alpha-1)}{n-1}$ and $\frac{2n}{n+1}<\alpha\le n$, $n\ge 2$.
Then, the inequality
\begin{equation}\label{2e-WRestrIneq}
\left(\int_{\Sp^{n-1}} |\h{f}(\omega)|^2\,d\sigma(\omega)\right)^{\frac{1}{2}}
\le C\left(\int_{\R^n} |f(x)|^2 V(x)\,dx\right)^{\frac{1}{2}},
\end{equation}
holds with $C=C'\|V\|_{\alpha,r}^{\frac{1}{2}}$.
\end{Thm}

In fact, in \cite{RV} it is proved that
\[
\|\h{d\sigma}*f\|_{L^2(V)}\le C'\|V\|_{\alpha, r}\|f\|_{L^2(V^{-1}\R^n)}
\]
but we can use Lemma \ref{L-duality} to shows that this inequality is equivalent to
\eqref{2e-WRestrIneq}. See also \cite{B2}.

Special cases of the restriction inequality in \cite{RV} are in \cite{CS} and
\cite{CR}. F.~Chiarenza and A.~Ruiz have proved in~\cite{CR} a version of
\eqref{2e-WRestrIneq} with special doubling weights; S.~Chanillo and E.~Sawyer
have proved in~\cite[Cor. 2.8]{CS}, that \eqref{2e-WRestrIneq} holds when $V$
is in the Fefferman--Phong class $F_r$, with $r\ge \frac{n-1}{2}$. In
particular, \eqref {2e-WRestrIneq} holds when $V^{-1}\in
L^{\frac{n-1}{2}}(\R^n)$.

\section{New Pitt inequalities}

In this section we obtain new Pitt-type inequalities for the Fourier transforms
using restriction inequalities from Section 2.

\begin{Thm}\label{1T-Restr-Pitt}
Assume that the restriction inequality \eqref{1e-wRestrIneq} holds for some
$1\le p\le q\le \infty$. Let $w(\rho)$ be a measurable function for which
$v(\rho x)\le w(\rho)v(x)$ for a.e. $\rho>0$ and $x\in\R^n$.
Suppose that $u$ is radial, and $u(x)=u_0(|x|)$ satisfies
\begin{equation}\label{1e-Cond-on-C}
\int_0^\infty \rho^{n-1-\frac{qn}{p'}} u_0(\rho) w^{\frac{q}{p}}(\rho)\,d\rho <\infty.
\end{equation} Then,
\begin{equation}\label{3e-NewPitt2}
\left(\int_{\R^n } |\h{f}(x)|^q U\Bigl(\frac{x}{|x|}\Bigr) u(x) \,dx\right)^{\frac{1}{q}}
\le C\|v^{\frac{1}{p}} f\|_{p}.
\end{equation}
\end{Thm}

Theorem \ref{new} is an easy consequence of Theorem \ref{1T-Restr-Pitt} (with
$U\equiv v\equiv 1$) and the Fefferman--Tao restriction theorem.

In the next section we will show that our theorem can be applied in cases where
prior results are not applicable.

Our next result deals with piecewise power weight $v$ defined by \eqref{p-w1}. In order to use
Theorem~\ref{1T-Restr-Pitt}, we need to find $w(\rho)$ so that $v(\rho x)\le
w(\rho)v(x)$, $\rho>0$. A straightforward calculation shows that in this case
\begin{equation}\label{maj}
w(\rho)\le w_{0}(\rho):=\max\{\rho^\alpha,\ \rho^\beta\}.
\end{equation}

Using Theorem \ref{1T-Restr-Pitt} and weighted restriction inequalities from
\cite{BS} (see~Section~3), we have

\begin{Cor}\label{1T-Rest-Bloom}
Let $1 < p\le 2 $ and $2\le q\le \frac{n-1}{n+1}\,p'$, with $n\ge 2$. Let $v$ be a
piecewise power weight $v(x)$ given by \eqref{p-w1} with $\alpha <n(p-1)$ and
$\beta \ge 0$. Let $u$ be a radial weight that satisfies
\[
\int_{0}^\infty\rho^{n-1-\frac{qn}{p'}} u_0(\rho)w_{0}^{q/p} (\rho)\,d\rho <\infty,
\]
where $w_{0}$ is given by \eqref{maj}. Then, for every $f\in C^\infty_0(\R^n)$,
\begin{equation}\label{1e-Pitt-Bloom}
\|u^{\frac{1}{q}} \h{f}\|_q\le C\|v^{\frac{1}{p}} f\|_p.
\end{equation}
\end{Cor}

\begin{Rem}
This corollary is valid for all piecewise power weights
$v$ and exponents $p,\ q$ for which the restriction theorems in \cite{BS} hold.
\end{Rem}

The following result uses weights in a Campanato--Morrey
class~$\cL^{\alpha,r}$ (see Section~2 for a definition).

\begin{Cor}\label{1T-Rest-Camp}
Let $V \in \cL^{\alpha,r}$, with $\frac{2n}{n+1}<\alpha\le n$ and
$\frac{\alpha}{n}\le \frac{1}{r}<\frac{2(\alpha-1)}{n-1}$, $n\ge 2$. Assume that there
exists a measurable function $w(\rho)$ for which $V(\rho x)\le w(\rho)V(x)$ for
a.e. $\rho>0$ and $x\in\R^n$,
and that $ u(x)=u_0(|x|)$ satisfies
\begin{equation}\label{2-p=q=2}
\int_0^\infty \rho^{-1 } u_0(\rho) w (\rho)\,d\rho <\infty.
\end{equation}
Then, for every $1\le p\le 2 $, the following weighted Hausforff-Young inequality holds
\begin{equation}\label{1e-Pitt-Camp}
\|u^{\frac{2}{p'}} \h{f}\|_{p'} \le C\| V^{\frac{2}{p'}}f\|_p
\quad f\in C^\infty_0(\R^n).
  \end{equation}
\end{Cor}

\section{Comparison of Theorems \ref{heinig} and \ref{new}}

In this section we give an example of radial weight $u(x)=u_0(|x|)$ that
satisfies the conditions of Theorem \ref{new} while does not satisfy the
conditions \eqref{1e-weight-cond} in Theorem \ref{heinig}.

We recall that Theorem \ref{new} states the Pitt inequality \begin{equation}\label{pp}
\|u^{\frac{1}{q}}\h{f}\|_q\le C\|f\|_p
\end{equation} holds with $1\le q\le \frac{n-1}{n+1}\,p'$ and
$ 1\le p < \frac{2(n+2)}{n+4}$ whenever $u(x)=u_0(|x|)$ satisfies
\begin{equation}\label{our}
\int_{0}^\infty \rho^{-a}u_0(\rho)\rho^{n-1}\,d\rho<\infty,\quad
a=\frac{qn}{p'}>0.
\end{equation}
On the other hand, when $u$ is radial and $v\equiv 1$, the sufficient condition
\eqref{1e-weight-cond} in Theorem \ref{heinig} states that
\[
\int_0^s u^*(t)dt\le C s^{\frac{q}{p'}}.
\]
The latter is equivalent to the following condition:
\begin{equation}\label{E-cond}
\sup_E |E|^{-\frac{q}{p'}}\int_{E}u\,dx<C,
\end{equation}
where supremum is taken over all measurable $E$,
$|E|>0$.

Let now $E_{0}$ be a measurable subset of $\R_{+}$. Consider the radial set
$E=\{x\in \R^{n}\colon |x|\in E_{0}\}$. For such set, we can rewrite
\eqref{E-cond} as follows:
\begin{equation}\label{E0-cond}
\int_{E_{0}}u_{0}(\rho)\rho^{n-1}\,d\rho\le
C|E|^{\frac{q}{p'}}=C\left(\int_{E_{0}}\rho^{n-1}\,d\rho\right)^{\frac{q}{p'}}.
\end{equation}
Let $A=\cup_{k=1}^{\infty}A_{k}$, where $A_{k}=(k,k+k^{-n-1})$. Set
\begin{equation}\label{weight}
u_{0}(\rho)\rho^{n-1}=\sum_{k=1}^{\infty}k^{n}\chi_{A_{k}}(\rho).
\end{equation}
Then condition \eqref{our} holds (and so also the Pitt inequality \eqref{pp}) since
\begin{align*}
\int_{0}^\infty \rho^{-a}u_0(\rho)\rho^{n-1}\,d\rho&=
\sum_{k=1}^{\infty}k^{n}\int_{k}^{k+k^{-n-1}}\rho^{-a}\,d\rho
\\
&\le \sum_{k=1}^{\infty}k^{n}k^{-a}k^{-n-1}=\sum_{k=1}^{\infty}k^{-1-a}<\infty
\end{align*}
and $a>0$.

On the other hand, taking $E_N=\{x\in \R^{n}\colon |x|\in \cup_{k=1}^{N}A_{k}\}$, we get
\[
\int_{E_{N}}\rho^{n-1}\,d\rho=\sum_{k=1}^{N}\int_{k}^{k+k^{-n-1}}\!\!\!\rho^{n-1}\,d\rho\le
\sum_{k=1}^{\infty}(k+1)^{n-1}k^{-n-1}<C.
\]
However,
\begin{equation}\label{u0-int}
\int_{E_{N}}u_{0}(\rho)\rho^{n-1}\,d\rho=
\sum_{k=1}^{N}k^{n}\int_{k}^{k+k^{-n-1}}\,d\rho
\asymp \ln N.
\end{equation}
Therefore, \eqref{E0-cond} (and so also \eqref{E-cond}), do not hold as $N\to \infty$.

\medskip
It is worthwhile to remark that for the radial weights $u$, the
necessary condition \eqref{1e-weight-cond-2} for the Pitt inequality
\eqref{our} to hold (see Theorem \ref{nec}) can be written as
\[
\sup_{s>0}\left(\int_{0}^{s}u_{0}(\rho)\rho^{n-1}\,d\rho\right)^{\frac 1 q}
\left(\int_{0}^{c_n/s}\rho^{n-1}\,d\rho\right)^{\frac 1 {p'}}<C
\]
or, equivalently,
\begin{equation}\label{u0-int-s-cond}
\sup_{s>0} s^{-a}\int_{0}^{s}u_{0}(\rho)\rho^{n-1}\,d\rho< C,
\end{equation}
where $a=\frac{qn}{p'}>0$. For the weight $u$ given by \eqref{weight} it can be easily
checked since
\[
\int_{0}^{s}u_{0}(\rho)\rho^{n-1}\,d\rho\le \sum_{k=1}^{[s]+1}k^{-1}\leq 1+\ln{}(s+1).
\]
This of course implies \eqref{u0-int-s-cond} since we only have to consider
the case $s\to \infty$.

\section{Proofs of the main results}


\begin{proof}[Proof of the Theorem \ref{nec}]
Let us assume that Pitt inequality \eqref{1e-Genpitt} hold.

\smallbreak
(1) \ Following \cite{H}, consider the function $f=\chi_{A}v^{1-p'}\in
L^{p}(v)$. For any set $B\subset \R^{n}$ we get
\[
C\|v^{\frac{1}{p}} f\|_p\ge\|u^{\frac{1}{q}} \h{f}\|_q\ge
\left(\int_{B}|\h{f}(\xi)|^{q}u(\xi)\,d\xi\right)^{\frac 1 q},
\]
where
\[
\|v^{\frac{1}{p}}f\|_p=\left(\int_{A}(v^{1-p'}(x))^{p}v(x)\,dx\right)^{\frac 1 p}=
\left(\int_{A}v^{1-p'}(x)\,dx\right)^{\frac 1 p}>0
\]
and
\[
|\h{f}(\xi)|\ge\left|\int_{A}v^{1-p'}(x)\cos{}(x\xi)\,dx\right|,\quad \xi\in B.
\]
Let $B=c_{n}A^{*}$, where $c_{n}<\pi/2$ and $A^{*}$ is polar set of the set~$A$.
Then for any $x\in A$ and $\xi\in B$ we have $|x\xi|\leq c_{n}$ and
$\cos{}(x\xi)\ge \cos c_{n}>0$. Therefore,
\[
|\h{f}(\xi)|\ge \cos c_{n}\int_{A}v^{1-p'}(x)\,dx,\quad \xi\in B.
\]
Hence,
\begin{align*}
C\left(\int_{A}v^{1-p'}(x)\,dx\right)^{\frac 1 p}&\ge
\left(\int_{B}|\h{f}(\xi)|^{q}u(\xi)\,d\xi\right)^{\frac{1}{q}}\\ &\ge \cos
c_{n}\left(\int_{A}v^{1-p'}(x)\,dx\right)\left(\int_{B}u(\xi)\,d\xi\right)^{\frac 1 q},
\end{align*}
or, equivalently,
\[
\left(\int_{c_{n}A^{*}}u(\xi)\,d\xi\right)^{\frac 1 q}\left(\int_{A}v^{1-p'}(x)\,dx\right)^{1/p'}<C.
\]

(2) \ If both weights $u$ and $v$ are radial, then the function
$f=\chi_{A}v^{1-p'}$ and its Fourier transform are also radial. Moreover,
taking the ball $A=sB^{n}$, we get
\[
\h{f}(\xi)=\omega_{n-1}\int_{A}v^{1-p'}(x)j_{n/2-1}(|\xi|x)\,dx.
\]
Let $q_{n/2-1}$ be the first zero of the normalized Bessel function $j_{n/2-1}(t)$. Note
that $q_{n/2-1}\ge q_{-1/2}=\pi/2$ and $q_{n/2-1}\sim n/2$ for $n\ge 1$. Then
$j_{n/2-1}(t)\ge j_{n/2-1}(c_{n})$, where $c_{n} \ge t$ can be taken as follows:
$\pi/2<c_n<q_{n/2-1}$ for $n\ge 2$. The rest of the proof is the same as
in~(1).

\smallbreak
(3) \ To prove this part, we use ideas similar to \cite{berndt}. In order to
consider translations of the sets $A$ and $c_{n}A^{*}$ by the vectors $x_{0}$
and $\xi_{0}$ correspondingly, it is enough to consider the function
$g(x)=f(x-x_{0})e^{-ix\xi_{0}}$ so that $|g(x)|=|f(x-x_{0})|$ and
$|\h{g}(\xi)|=|\h{g}(\xi-\xi_{0})|$. The integral condition \eqref{nec-cond} easily applies to unions of disjoint
translations of $A$ and $cA^*$.
\end{proof}

\begin{proof}[Proof of Lemma \ref{L-duality}]
Let $A: L^{p}(v\,dx)\to
 L^{q}(\Sp^{n-1})$ be the operator, initially defined for all $f\in C^\infty_0(\R^n)$, by
   $ A f(\omega)= \h{f}(\omega)U^{\frac{1}{q}}(\omega). $ Duality gives
\begin{align*}
\|Af\|_{L^{q}(\Sp^{n-1})} & =
\sup_{\|g\|_{L^{q'}(\Sp^{n-1})}\le 1}\left\vert
\int_{
\Sp^{n-1}
}Af(\omega)g(\omega)\,d\sigma(\omega) \right\vert
\\
&=
\sup_{\|g\|_{L^{q'}(\Sp^{n-1})}\le 1}\left\vert
\int_{
\R^n
}f(x)A^*g(x)\,dx \right\vert,
\end{align*}
where \begin{equation}\label{e2-A*} A^*g(x)=
\int_{\Sp^{n-1}}g(\omega)e^{i\omega x}U^{\frac{1}{q}}(\omega)\,d\sigma(\omega)
.\end{equation}
By H\"{o}lder's inequality
\[
\int_{\R^{n}}f(x)A^{*}g(x)\,dx \le
\|v^{\frac{1}{p}}f\|_{p}\|v^{-\frac{1}{p}}A^{*}g\|_{p'}=\|v^{-\frac{1}{p}}A^{*}g\|_{p'}\|f\|_{L^{p}(v\,dx)}.
\]
Therefore, the inequality
\[
\|v^{-\frac{1}{p}}A^{*}g\|_{p'}=
\left\| \int_{\Sp^{n-1}}g(\omega)e^{i\omega x}U^{\frac{1}{q}}(\omega)d\sigma(\omega)\right\|_{L^{p'}(v^{1-p'}dx)} \!
\leq C\|g\|_{L^{q'}(\Sp^{n-1})}
\]
implies \eqref{2eRestrIneq-w}. A similar argument shows that the inequality
\eqref{2eRestrIneq-w}, or $\|Af\|_{L^{q}(\Sp^{n-1})}\leq
C\|f\|_{L^{p}(v\,dx)},$ implies
\begin{equation}\label{2e-dual-A*}
\|v^{-\frac{1}{p}}A^{*}g\|_{p'} \leq C \|g\|_{L^{q'}(\Sp^{n-1})} .
\end{equation}

\end{proof}

\begin{proof}[Proof of Proposition \ref{P-nec1}]
Let $A$ and $A^*$ be defined as in Lemma~\ref{L-duality}.
Let
$g(\omega)=U^{-\frac{1}{q}}(\omega)$. Clearly, $g\in L^{q'}( \Sp^{n-1})$, and by \eqref{e2-A*}
\[ \|v^{-\frac{1}{p}}A^{*}g\|_{p'}=
A^{*}g(x)=\int_{\Sp^{n-1}}e^{i\omega
x}\,d\sigma(\omega)=\omega_{n-1}j_{n/2-1}(|x|),
\]
(see e.g. \cite{St}).
From \eqref{2e-dual-A*} it follows that
\begin{equation}\label{v-j-cond-1}
\int_{\R^n} v^{1-p'}(x)|j_{n/2-1}(|x|)|^{p'}\,dx\le
C\left(\int_{\Sp^{n-1}}U^{1-q'}(\omega)\,d\sigma(\omega)\right)^{\frac{p'}{q'}}
\end{equation}
as required.
\end{proof}

\begin{proof}[Proof of Corollary \ref{C-nec2}]
Let $ q_{k}=q_{\alpha,k}$, ${k\ge 1}$, be the positive zeros of the Bessel
function $J_{\alpha}(t)$ in nondecreasing order. It is known (see e.g.
\cite{W}) that
\[
J_{\alpha}(t)=C_{\alpha}t^{-1/2}\left(\cos{}(t-c_{\alpha})+O( t ^{-1})\right)
\]
as $t\to +\infty$.
This gives $|j_{\alpha}(t)|\le
C(1+t)^{-\alpha-1/2}$, $t\ge 0$, and
\begin{equation}\label{v-j-cond-2}
|j_{\alpha}(t)|\ge C(1+t)^{-\alpha-1/2},\quad t\in
I:=[0, \infty)- \mathop{\text{\Large$\cup$}}\limits_{k=1}^\infty I'_k
\end{equation}
where $I'_k=(q_k-\varepsilon, q_k+\varepsilon)$ and
$\varepsilon=\varepsilon_{\alpha}>0$ is chosen so that
$ I'_k\cap I'_{l}=\emptyset$ when $k\ne l$.
We let $I:= \cup_{k=0}^\infty I_k$
 and $I_k= [a_k, b_k]$, with $I_{0}=[0,q_{1}-\varepsilon]$ and
$I_{k}=[q_{k}+\varepsilon,q_{k+1}-\varepsilon]$.

It is well known that $q_{k}\sim \pi k$, and there exist constants $c_i>0,$
$i=1,\ldots,4$, that depend only on $\alpha=n/2-1$ so that $c_1\le
q_{k+1}-q_{k}\le c_2$ and, when $k\ne 0$, $c_3\leq
|I_{k}|=q_{k+1}-q_{k}-2\varepsilon\leq c_4$.

Inequalities \eqref{v-j-cond-1} and \eqref{v-j-cond-2} give
\[
\int_{|x|\in I}v^{1-p'}(x)(1+|x|)^{-\frac{p'(n-1)}{2}}\,dx<C.
\]

Furthermore,
\begin{align*}
J&:=\omega_{n-1}^{-1}\int_{\R^{n}}v^{1-p'}(x)(1+|x|)^{-\frac{p'(n-1)}{2}}\,dx\\
&=\int_{0}^{\infty}v_0^{1-p'}(t)(1+t)^{-\frac{p'(n-1)}{2}}t^{n-1}\,dt=
\int_{I_{0}}+\sum_{k=1}^{\infty}\left(\int_{I_{k}}+\int_{I_{k}'}\right).
\end{align*}
Assume that condition \eqref{vs1} holds. Then it is clear that
\[
\int_{I_{k}'} v_{0}^{1-p'}(t)\,dt\le C\int_{I_{k}} v_{0}^{1-p'}(t)\,dt
\]
with some constant $C$. Using this, we get
\begin{align*}
&\int_{I_{k}'}v_{0}^{1-p'}(t)(1+t)^{-\frac{p'(n-1)}{2}}t^{n-1}\,dt
\\
&\qquad \le C
(1+b_{k-1})^{-\frac{p'(n-1)}{2}}a_{k}^{n-1}\int_{I_{k}'}v_{0}^{1-p'}(t)\,dt
\\
&\qquad \le C
(1+b_{k})^{-\frac{p'(n-1)}{2}}\int_{I_{k}}v_{0}^{1-p'}(t)t^{n-1}\,dt
\\
&\qquad \le C
\int_{I_{k}}v_{0}^{1-p'}(t)(1+t)^{-\frac{p'(n-1)}{2}}t^{n-1} \,dt,
\end{align*}
since $b_{k}=b_{k-1}+|I_{k}'|+|I_{k}|\le b_{k-1}+c\le Cb_{k-1}$.

Thus, 
\begin{align*}
J&= \int_{I_{0}}+ \sum_{k=1}^{\infty} \left( \int_{I_{k}}+\int_{I_{k}'} \right)
v_0^{1-p'}(t)(1+t)^{-\frac{p'(n-1)}{2}}t^{n-1}\,dt
\\
& \le C
\sum_{k=0}^{\infty} \int_{I_{k}}
 \le C
\int_{|x|\in I}v^{1-p'}(x)(1+|x|)^{-\frac{p'(n-1)}{2}}\,dx<C.
\end{align*}
If the condition \eqref{vs2} is satisfied, the proof is similar.

\end{proof}

We prove Lemma \ref{2L-restr-gen-CS} to make the paper self-contained.

\begin{proof}[Proof of Lemma \ref{2L-restr-gen-CS}]
Assume $s>p$, since the proof in the other case is similar. Let
$r=\frac{p}{s-p}$. Suppose that $T\colon L^p(d\mu)\to L^q(d\nu)$ is bounded. To
show that $T\colon L^s(wd\mu)\to L^q(d\nu)$ is bounded, we observe that
$\frac{1}{rs}= \frac{s-p}{sp}=\frac{1}{p}-\frac{1}{s}$. By H\"older's
inequality,
\begin{align*}
\|Tf\|_{L^q(d\mu)}&\le C\|f\|_{L^p(d\mu)} = C\|w^{-\frac{1}{s}}w^{\frac{1}{s}}
f\|_{L^p(d\mu)}
\\
&\le C\|w^{-\frac{1}{s}}\|_{L^{rs}(d\mu)}\|w^{\frac{1}{s}} f\|_{L^s(d\mu)}=
C\|w^{-1}\|_{{L^r(d\mu)}}^{\frac{1}{s}}\|w^{\frac{1}{s}} f\|_{L^s(d\mu)},
\end{align*}
as required.

To prove the other direction we argue as \cite{CS} and as in the proof of
Proposition~1.10 in~\cite{BS}. Observe that
\[
\|w^{\frac{1}{s}} f\|_{L^s(d\mu)}^{s} = \int_{\R^n} w|f(x)|^s\,d\mu(x) =
\int_{\R^n} |f(x)|^p\,d\mu(x)
\]
with $w=|f|^{p-s}$. Since
\[
\|w^{-1}\|_{L^r(d\mu)}^{\frac{1}{s}}=\left(\int_{\R^n}
|f(x)|^p\,d\mu(x)\right)^{\frac{s-p}{sp}}=\|f\|_{L^p(d\mu)}^{1-\frac{p}{s}},
\]
we obtain
\begin{align*}
\|Tf\|_{L^q(d\nu)} &\le C\|w^{-1}\|_{L^r(d\mu)}^
{\frac{1}{s}}\|f\|_{L^p(d\mu)}^{\frac{p}{s}}=
C\|f\|_{L^p(d\mu)}^{1-\frac{p}{s}}\|f\|_{L^p(d\mu)}^{\frac{p}{s}} \\ &=
C\|f\|_{L^p(d\mu)}.
\end{align*}
\end{proof}

\begin{proof}[Proof of the Theorem \ref{1T-Restr-Pitt}]
Fix $\rho>0$ and $f\in C^\infty_0(\R^n)$; let $\delta_\rho\psi(x) =\psi(\rho
x)$, and let $g(x)= \rho^{-n}\delta_{\frac{1}{\rho}} f(x)$. We apply
\eqref{1e-wRestrIneq} with $g$ in place of $f$. Recalling that
$\rho^{-n}\h{\delta_{\frac{1}{\rho}} f}= \delta_\rho \h{f}$, we obtain by
Lemma~\ref{2L-restr-gen-CS} and \eqref{1e-wRestrIneq} with $d\nu=U\,d\omega$
and~$d\mu=v\,dx$:
\begin{align*}
\left(\int_{\Sp^{n-1}} |\delta_\rho \h{f}(\omega)|^q U(\omega)
d\sigma(\omega)\right)^{\frac{1}{q}} &=\left(\int_{\Sp^{n-1}} |\widehat
g(\omega)|^q U(\omega)\,d\sigma(\omega)\right)^{\frac{1}{q}} \\ &\le
C\|g\|_{L^p(v\,dx)} = C\rho^{-n}\|v^{\frac{1}{p}}\delta_{\frac{1}{\rho}}f\|_p
\\ &= C\rho^{-n+\frac{n}{p}}\|(\delta_\rho v)^{\frac{1}{p}} f\|_p.
\end{align*}
By our assumptions on $v$ we obtain
\begin{equation}\label{3e-summary}
\int_{\Sp^{n-1}}|\h{f}(\rho\omega)|^q U(\omega)\,d\sigma(\omega)\le C
\rho^{-\frac{nq}{p'}}w^{\frac{q}{p}}(\rho)\|v^{\frac{1}{p}} f\|_p^q.
\end{equation}
We multiply both sides of this inequality by $ u_0(\rho) \rho^{n-1} $ and we
integrate with respect to $\rho$. We obtain
\begin{align*}
&\int_{0}^\infty \rho^{n-1}\int_{\Sp^{n-1}} |\h{f}(\rho\omega)|^q
u_0(\rho)U(\omega)\,d\sigma(\omega)\,d \rho \\ &\qquad\le C\int_0^\infty
\rho^{n-1-\frac{nq}{p'}} u_0(\rho) w^{\frac{q}{p}}(\rho)\,d\rho
\,\|v^{\frac{1}{p}} f\|_p^q
\end{align*}
which by \eqref{1e-Cond-on-C} implies $\int_{\R^n } U(\frac{x}{|x|}) u(x) |\h{f}(x)|^q\,dx
\le C\|v^{\frac{1}{p}} f\|_p^q$.
\end{proof}

\begin{proof}[Proof of Corollary \ref{1T-Rest-Camp}]
When $p=q=2$, we use Theorem \ref{T-RuizV}. The assumptions of Theorem
\ref{1T-Restr-Pitt} are satisfied, and so the following inequality holds:
\begin{equation}\label{eq-R2}
\|\h{f}\|_{L^2(u\,dy)}\le\|f\|_{L^2(V\,dx)}.
\end{equation}

To conclude the proof of Corollary \ref{1T-Rest-Camp} we use a special case of
an interpolation theorem with change of measure proved in \cite{SW}.

\begin{Lemma}\label{L-SW}
Let $Tf$ be a linear operator defined in a space of measurable functions that
include $L^{p_1}(V_1dx)$ and $L^{p_2}(V_2dx)$; assume that
\[
\|Tf\|_{L^{q_1}(u_1\,dy)}\le C\|f\|_{L^{p_1}(V_1\,dx)}
\quad \textup{and}
\quad\|Tf\|_{L^{q_2}(u_2\,dy)}\le C\|f\|_{L^{p_2}(V_2\,dx)}.
\]
Then, for every $0\le t\le 1$,
\begin{equation}\label{e-stW}
\|Tf\|_{L^{q_t }(u_1^tu_2^{1-t}\,dy)}\le C\|f\|_{L^{p _t }(V_1^tV_2^{1-t}\,dx)}
\end{equation}
where $\frac{1}{p_t}=\frac{t}{p_1}+\frac{1-t}{p_2}$ and
$\frac{1}{q_t}=\frac{t}{q_1}+\frac{1-t}{q_2}$.
\end{Lemma}

We apply Lemma \ref{L-SW} with $Tf=\h{f}$; we interpolate the inequality
\eqref{eq-R2} and the $\|\h{f}\|_\infty\le\|f\|_1$; we let $u=u_1$ and $V=V_1$,
and $u_2= V_2= 1$; we let $
\frac{1}{p_t}=\frac{t}{2}+1-t=1-\frac{t}{2}$, so that $t=
2\left(1-\frac{1}{p_t}\right)=\frac{2}{p_t'}$. Note that $ q_t=p_t'$. By
\eqref{e-stW}, we have
\[
\|\h{f}\|_{L^{p'}(u^{\frac{2}{p'}}dy)}\le\|f\|_{L^{p}(V^{\frac{2}{p'}}dx)}
\]
where we have let $p=p_t$ for simplicity. That concludes the proof of the
corollary.
\end{proof}

\section{Applications to the uncertainty principle}

The uncertainty principle is a cornerstone in quantum physics and in Fourier
Analysis. The simplest formulation of the uncertainty principle in harmonic
analysis is \textit{Heisenberg's inequality}, which applies to functions in
$L^2(\R^n)$ of norm $=1$. It states that the product of the variances of $f$
and $\h{f}$ is bounded above by a universal constant, i.e.
\[
\inf_{a \in\R^n}\int_{\R^n} |x-a|^2 |f(x)|^2\,dx\,
\inf_{b \in\R^n}\int_{\R^n} |\xi-b|^2 |\h{f}(\xi)|^2\,d\xi\ge \frac{(2\pi)^n n^2}{4}.
\]
One of the many consequences of this inequality is that a nonzero function and
its Fourier transform cannot both be compactly supported.

The uncertainty principle for $L^p$ functions is also interesting. Inequalities
in the form of $\|f\|_2^2\leq
C\|v^{\frac{1}{p}}f\|_p\|w^{\frac{1}{q}}\h{f}\|_q$, where $v$ and $w$ are
suitable weight functions and $1\le p,\ q\le \infty$ are discussed in
\cite{CP}. Power weights are of particular interest:
using a standard homogeneity argument, is easy to prove that a necessary
condition for the inequality {$||f||_2^2\leq C \||x|^a f\|_p\||\xi|^b\h{f}\|_q\ $}
to hold for all $f\in C^\infty_0(\R^n)$ is that $a+\frac{n}{p}=b+\frac{n}{q}$. %
See also \cite{FS} for a survey on uncertainty principle.

\medskip
We prove the following
\begin{Thm}\label{Uncertainty}
Let $u,\ v$ be weights for which the Pitt inequality \eqref{1e-Genpitt} holds
for some $1\le p, q\le \infty$. Then, for every $f\in C^\infty_0(\R^n)$,
\[
\|f\|_2^2\le C\big\|u^{-\frac{1}{q}}|\xi|\h{f}\big\|_{q'}\big\|v^{\frac{1}{p}}|x|f\big\|_p,
\]
where $C$ is independent of $f$.
\end{Thm}

\medskip

\begin{Cor}\label{C-UncPrin}
  Let $1\le p < \frac{2(n+2)}{n+4}$ and $1\le q\le
\frac{n-1}{n+1}\,p'$.
Let $s (x)= s_0(|x|)$ be a radial weight that satisfies
\begin{equation}\label{int-fin}
\int_{0}^\infty \frac{\rho^{n-1-\frac{qn}{p'} }}{s_0 (\rho)}\,d\rho <\infty.
\end{equation}
Then,
\begin{equation}\label{ineq3}
\|f\|_2^2\le C\big\|s_0^{\frac 1q}(|\xi|)|\xi|\ \h{f}\big\|_{q'}\big\||x|f\big\|_p, \quad
f\in C^\infty_0(\R^n).
\end{equation}

\end{Cor}

For example, $s_0(\rho)= \rho^{-m} (1+\rho)^{m+n- \frac{nq}{p'} +\varepsilon}$,
with $\varepsilon>0,$ and $m +n -\frac{qn}{p'}>0$, satisfies (\ref{int-fin}).

\medskip

\begin{Cor}\label{C-UncPrin2}
Let
\[
v (x)=\begin{cases} |x|^{\alpha }, & |x|\le 1,\\ |x|^{\beta }, & |x|>1,\end{cases}\quad
\text{and}\quad w_0 (\rho)= \max\{\rho^\alpha,\ \rho^\beta\},
\]
Let $1 < p\le 2$, $2\le q\le \frac{n-1}{n+1}\,p'$,
 $\alpha <n(p-1)$, and $\beta \ge 0$. We have
\begin{equation}\label{ineq4}
\|f\|_2^2\le C\big\|s_0^{\frac 1q}(|\xi|)|\xi| \ \h{f}\big\|_{q'}\big\|\,|x| v^{\frac{1}{p}}f\big\|_p,\quad
f\in C^\infty_0(\R^n).
\end{equation}
provided
\begin{equation}\label{int-fin2}
\int_{0}^\infty \frac{\rho^{n-1-\frac{qn}{p'} }w_0 ^{\frac qp}(\rho)}{s_0 (\rho)}\,d\rho <\infty.
\end{equation}
When $\alpha<n$ and $\beta>1$, we have
\begin{equation}\label{ineq5}
\|f\|_2^2\le C\big\|s_0^{\frac{1}{2}}(|\xi|)|\xi| \h{f}\big\|_{2} \big\|\, |x| v^{\frac{1}{2}} f\big\|_2,\quad
f\in C^\infty_0(\R^n),
\end{equation}
provided
\[
\int_{0}^\infty \frac{ w_0 (\rho)}{\rho\,s_0 (\rho)}\,d\rho <\infty.
\]
\end{Cor}

\begin{proof}[Proof of Theorem \ref{Uncertainty}]
We use the same idea of the proof of the $L^2$ Heisenberg principle (see
\cite{FS}). Let $f\in C^\infty_0(\R^n)$. We denote $x=(x_1, \dots ,
x_n)\in\R^n$ by $(x_1, x')$, with $x'\in\R^{n-1}$. We integrate by parts the
function $ |f(x)| ^2 = |f(x_1, x')|^2$ with respect to $x_1$. That is,
\[
\int_{-\infty}^\infty |f(x_1, x')|^2\,dx_1 = x_1|f(x_1, x')|^2 \bigr|_{x_1=-\infty}^\infty -
\int_{-\infty}^\infty x_1 \frac{\partial\,|f(x_1, x')|^2}{\partial x_1}\,dx_1.
\]

A simple calculation shows that
\[
\frac{\partial\,|f(x_1, x')|^2}{\partial x_1}=
\frac{\partial}{\partial x_1}\left(f(x_1, x')\overline{f(x_1, x')}\right)=
2\Re\left(\overline{f(x_1, x')}\,\frac{\partial f(x_1, x')}{\partial x_1}\right).
\]
We obtain
\[
\int_{-\infty}^\infty |f(x_1, x')|^2\,dx_1=
-2\Re\int_{-\infty}^\infty x_1\overline{f(x_1, x')}\,\frac{\partial f(x_1, x')}{\partial
x_1}\,dx_1.
\]
We integrate the above identity in $x'$, to obtain
\[
\|f\|_2^2=-2\Re\int_{\R^n}x_1\overline{f(x)}\,\frac{\partial f(x)}{\partial x_1}\,dx.
\]
We use the identity
$\int_{\R^{n}}f_{1}\overline{f_{2}}\,dx= (2\pi)^{-n} \int_{\R^{n}}\h{f}_{1}\,\overline{\h{f}_{2}}\,d\xi$,
and we recall that the Fourier transform of $\frac{\partial f(x)}{\partial
x_1}$ is $-i\xi_1\h{f}(\xi)$. Thus,
\begin{align*}
\|f\|_2^2 &= 2(2\pi)^{-n}\Re \left( i
\int_{\R^n}\xi_1\h{f}(\xi)\,\overline{(\h{x_1f})(\xi)}\,d\xi\right) \\ & =
2(2\pi)^{-n}\Re \left( i \int_{\R^n}(u^{-\frac{1}{q}}\xi_1
\h{f}(\xi))\overline{(\,u^{\frac{1}{q}} \, (\h{x_1f}) (\xi)} \,d\xi \right)
\end{align*}
and by H\"older inequality and Theorem \ref{1T-Restr-Pitt},
\begin{align*}
\|f\|_2^2 &\le C\,\|u^{-\frac{1}{q}} \xi_1\h{f}\|_{q'}\|u^{\frac{1}{q}}
\h{x_1f}\|_q
\\
&\le C\|u^{-\frac{1}{q}} \xi_1\h{f}\|_{q'}\|v^{\frac{1}{p}} x_1 f\|_p
\\
&\le C\|u^{-\frac{1}{q}} |\xi|\h{f}\|_{q'}\|v^{\frac{1}{p}} |x| f\|_p
\end{align*}
as required.
\end{proof}

\begin{proof} [Proof of Corollary \ref{C-UncPrin}]
Follows from Theorems \ref{new} and \ref{Uncertainty}, with $v\equiv 1$ and
$u_0(\rho)= s_{0}^{-1}(\rho)$.
\end{proof}

\begin{proof}[Proof of Corollary \ref{C-UncPrin2}]
follows from Corollary \ref{1T-Rest-Bloom} and Theorem \ref{Uncertainty}, with
$u_0(\rho)= s_{0}^{-1}(\rho)$.
\end{proof}

\section{Riemann--Lebesgue estimates via Pitt inequalities}

Here we investigate the interrelation between the smoothness of a function and
the growth properties of the Fourier transforms. The original result goes back
to the Riemann--Lebesgue estimate $|\widehat{f}(\xi)|\to 0$ as $|\xi|\to
\infty$, where $f\in L^1(\mathbb{R}^n)$ and its quantitative version given by
\begin{equation}\label{fc}
|\widehat{f}(\xi)|\le C \omega_l\left(f,\frac{1}{|\xi|}\right)_1, \quad f\in L^1(\R^n),
\end{equation}
where the modulus of smoothness $\omega_l(f,\delta)_p$ of a function $f\in L^{p}(X)$
is defined by
\begin{equation}\label{mod}
\omega_l\left(f,\delta\right)_{p} = \sup_{|h|\le\delta}
\left\| \Delta^l_h f (x) \right\|_{L^p(\R^n)},\quad 1\le p \le \infty,
\end{equation}
and
\[
\Delta^l_h f (x) = \Delta^{l-1}_h\left(\Delta_h f (x)\right),
\qquad \Delta_h f (x)=f(x+h)-f(x).
\]
Recently this result was extended for $L^p$-functions. Let us first define the
suitable multivariate substitution for the classical modulus of smoothness.

For a locally integrable function~$f$ the average on a sphere in $\R^{n}$ of
radius $t>0$ is given~by
\[
V_t f(x):=\frac{1}{m_t}\int_{|y-x|=t} f(y)\,dy\quad \textup{with}\quad V_t
1=1,\quad n\ge 2.
\]
For $l\in \N$ we define
\[
V_{l,t}f (x):=\frac{-2}{\binom{2l}{l}} \sum_{j=1}^l (-1)^j
\binom{2l}{l-j} V_{jt} f(x).
\]
and set
\[
\Omega_{l}(f,t)_{p}=\|f - V_{l,t} f\|_p.
\]
In \cite[Th.~2.1 (A), $n\ge 2$]{GorTik12} the following Riemann--Lebesgue type
estimates was proved.

\begin{Thm}\label{GorTik12-thm}
Let $f\in L^{p}(\R^{n})$, $1<p\le 2$. Then for $p\le q\le p'$ we have
$|\xi|^{n(1-\frac 1 p-\frac 1 q)}\h{f}(\xi)\in L^{q}(\R^{n})$, and
\[
\left(\int_{\R^n}\left[\min{}(1,t|\xi|)^{2 l} |\xi|^{n(1-\frac 1 p-\frac 1 q)}
|\h{f}(\xi)|\right]^{q}\,d\xi\right)^{\frac{1}{q}}\le C\Omega_l(f,t)_p.
\]
\end{Thm}

Note that some partial cases were previously proved in \cite{BP, Di}; see
also~\cite{Br1}. The essential step in the proof of Theorem \ref{GorTik12-thm}
is the use of Pitt's inequalities \eqref{1e-pitt} under conditions
\eqref{1e-relation-ab} and \eqref{1e-relation-b} in the case when $b=0$, that
is when the right-hand side of \eqref{1e-pitt} is the non-weighted $L^p$-norm.

Here we refine Theorem \ref{GorTik12-thm} using new Pitt's inequality given by
Theorem~\ref{new}.

\begin{Thm}
Under the assumption of Theorem~\textup{\ref{new}}, we have
\[
\left(\int_{\R^n}\left[\min{}(1, \ t|\xi|)^{2 l}
|\h{f}(\xi)|\right]^{q}u(\xi)\,d\xi\right)^{\frac 1 q}\le C\Omega_l(f,t)_p.
\]
\end{Thm}

The proof repeats the proof of Theorem~\ref{GorTik12-thm} with the only
modification that one should use the weight $u^{\frac{1}{q}}(\xi)$ in place of
$|\xi|^{n(1-\frac 1 p-\frac 1 q)}$ (see ~\cite[(2.16)]{GorTik12}) and
Theorem~\textup{\ref{new}}.

\section{Other applications}
Inequality \eqref{1e-cond-nu} in Theorem \ref{new} implies $\int_{\R^n} u(\xi)
(1+|\xi|)^{-\frac{qn}{p'}}d\xi <\infty$. In \cite{Bu} it is proved that if
\eqref{1e-Genpitt} holds for $1<p\le q<\infty$, and if
$\int_{\R^n} u(\xi)^{1-q'}(1+|\xi|)^{-M}d\xi <\infty$ for some $M>0$, then once
can prove a Bernstein-type theorem, which characterizes the Fourier transform
on weighted Besov spaces. We leave the generalization of the main Theorem in
\cite{Bu} to the interested reader.

\end{document}